\documentclass[11pt]{amsart}
\usepackage[T1]{fontenc}
\usepackage[utf8]{inputenc}
\usepackage{lmodern}
\usepackage{geometry}
\usepackage{amsmath,amssymb,amsthm}
\usepackage{mathtools}
\usepackage{microtype}
\usepackage{enumitem}
\usepackage{hyperref}
\usepackage{xcolor}

\usepackage{tikz}
\usetikzlibrary{decorations.pathreplacing,calc}
\usetikzlibrary{arrows.meta}
\usetikzlibrary{patterns}
\usetikzlibrary{shapes.misc}
\usetikzlibrary{arrows.meta,calc}

\tikzset{
    gridline/.style={gray!30, thin},
    nodepoint/.style={circle, fill=black, inner sep=1.2pt},
    halfface/.style={thick},
    cutface/.style={thick},
    interface/.style={thick},
    annotation/.style={font=\small}
}

\hypersetup{colorlinks=true, linkcolor=blue, citecolor=blue, urlcolor=blue}

\theoremstyle{plain}
\newtheorem{theorem}{Theorem}[section]
\newtheorem{lemma}[theorem]{Lemma}
\newtheorem{corollary}[theorem]{Corollary}
\newtheorem{proposition}[theorem]{Proposition}

\newtheorem{assumption}[theorem]{Assumption}
\newtheorem{definition}[theorem]{Definition}

\numberwithin{equation}{section}

\title[Second-Order Accuracy from Large Local Errors:
A Green's Function Analysis]
{Second-Order Accuracy from Large Local Errors in Interface Problems:
A Discrete Green's Function Analysis}

\author{So-Hsiang Chou}
\author{Patrick Nyadjo Fonga}
\begin{document}

\begin{abstract}
Finite difference methods for interface problems often exhibit large
local truncation errors near the interface, particularly when
discontinuities in coefficients or solution derivatives are present.
Nevertheless, many such schemes achieve second-order accuracy, a
phenomenon not fully explained by standard pointwise consistency
arguments.

In this work, we provide a precise explanation of this behavior
through the structure of the discrete Green's function associated
with the underlying conservative difference operator.  An explicit
representation of the Green's function reveals a two-plateau
structure in its weighted increments.  By expressing the numerical
error in terms of this Green kernel, we show that the dominant
interface truncation errors, although individually of order \(O(1)\),
possess a cancellation structure that reduces their effective
contribution to the global error.

The exact Green's-function analysis leads to a class of conservative interface flux--balance schemes for which the cancellation mechanism yields second-order accuracy in the maximum norm. The weighted harmonic discretization is included as a particular member, and its cancellation property is established directly. For Cartesian grids with a flat interface parallel to a coordinate direction, the same conservative cancellation mechanism persists along grid lines crossing the interface. Numerical experiments in one and two dimensions directly illustrate the predicted cancellation behavior and the resulting second-order accuracy.

These results demonstrate that second-order accuracy in interface
problems can arise from the interaction between localized truncation
errors and the global structure of the discrete operator, rather than
from pointwise consistency alone.
\end{abstract}

\maketitle

\noindent\textbf{keywords:}
discrete Green's function; interface problems; finite difference methods;
local truncation error; conservative discretization; maximum-norm convergence.
\medskip

\address{Department of Mathematics and Statistics,
Bowling Green State University,
Bowling Green, OH 43403, USA}

\email{Corresponding Author: chou@bgsu.edu}

\section{Introduction}
\label{sec:introduction}

Finite difference and finite volume methods for interface problems
often exhibit large local truncation errors near the interface,
especially when discontinuities in coefficients or solution
derivatives are present.  Nevertheless, many such schemes retain
second-order global accuracy.  This behavior cannot be explained by
pointwise consistency alone, since the truncation errors at grid
points adjacent to the interface may remain of order \(O(1)\).

A variety of numerical approaches have been developed to treat such
loss of regularity.  Immersed interface methods incorporate jump
conditions through local corrections; see, for example, LeVeque and
Li~\cite{leveque1994immersed} and Li and Ito~\cite{li2006immersed}.
Related approaches based on local stencil modification and algebraic
enforcement of interface conditions include the method of Wiegmann
and Bube~\cite{WiegmannBube2000}.  Simple finite-difference
constructions for elliptic interface problems have also been studied
by Tzou and Stechmann~\cite{tzou2019simple}.  Harmonic-type interface
schemes, such as those considered in \cite{pan2026harmonic}, provide
another attractive approach because of their conservative flux
structure and simple implementation.

A broader context is provided by the classical analysis of conservative
finite difference and finite volume schemes for problems with nonsmooth
solutions or discontinuous coefficients.  General finite volume
principles based on local conservation and numerical fluxes are
developed, for example, by Eymard, Gallou\"et, and
Herbin~\cite{EymardGallouetHerbin2000}; related cell-centered finite
volume analysis may be found in Lazarov, Mishev, and
Vassilevski~\cite{LazarovMishevVassilevski1996}.
Finite-difference convergence under reduced regularity is treated
systematically by Jovanovi\'c and S\"uli~\cite{JovanovicSuli2014}.  For elliptic
problems with discontinuous coefficients, Ewing, Iliev, and
Lazarov~\cite{EwingIlievLazarov2001} developed modified finite volume
approximations based on accurate treatment of the interface flux.
Harmonic averaging also arises naturally in one-dimensional
finite-volume difference discretizations of interface problems; see
Ewing, Iliev, Lazarov, and
Naumovich~\cite{EwingIlievLazarovNaumovich2007}.  From a different
direction, Beale and Layton~\cite{BealeLayton2006} showed how reduced
local consistency near an interface can nevertheless lead to
second-order solution accuracy through discrete elliptic and
Green's-function estimates.  These results provide a broader setting for the question considered here: how localized interface defects are transmitted through a conservative discrete operator. Our purpose is to isolate the underlying mechanism in a setting where the discrete Green's function can be obtained explicitly and the interaction between conservation, cancellation, and global error propagation can be seen without additional geometric effects. The resulting analysis leads to a structural cancellation principle for conservative interface discretizations and, for Cartesian grids with a flat interface, the same mechanism persists along grid lines crossing the interface.

We begin with the one-dimensional flux-form problem, where the Green's kernel can be characterized exactly. By examining how localized truncation errors propagate through the inverse operator, we show that conservation forces a cancellation between the leading interface defects and that the discrete Green's function converts this cancellation into second-order global accuracy.

To make this precise, consider the one-dimensional interface problem
\begin{equation}\label{eq:Dirichlet}
-\frac{d}{dx}
\left(
\beta(x)\frac{du}{dx}
\right)
=f,
\qquad
x\in(a,\alpha)\cup(\alpha,b),
\qquad
u(a)=u(b)=0,
\end{equation}
where
\[
\beta(x)>0
\]
may be discontinuous at \(x=\alpha\). For any quantity \(g\) admitting one-sided limits at \(\alpha\), we use
the convention
\[
[g]_\alpha:=g(\alpha^+)-g(\alpha^-).
\]  Across the interface \(\alpha\) we impose the jump conditions
\[
[u]_\alpha=w,
\qquad\text{and}\qquad
[\beta u_x]_\alpha=v.
\]
 Let
\[
a=x_0<x_1<\cdots <x_n<x_{n+1}=b,
\qquad
x_i=a+ih,
\qquad
h=\frac{b-a}{n+1}.
\]
We assume that the interface lies strictly between two consecutive grid
points,
\(
x_I<\alpha<x_{I+1},
\)
for some \(1\le I\le n-1\). Let \(U_i\) denote the numerical
approximation to \(u(x_i)\), with
\[
U_0=U_{n+1}=0.
\]

We consider the conservative flux-form system
\begin{equation}\label{eq:generalsys}
-\beta_{i-\frac12}U_{i-1}
+
\left(
\beta_{i-\frac12}+\beta_{i+\frac12}
\right)U_i
-
\beta_{i+\frac12}U_{i+1}
=
h^2 r_i^h,
\qquad i=1,\ldots,n,
\end{equation}
with
\[
\beta_{i+\frac12}>0,
\qquad i=0,\ldots,n.
\]
The quantities \(\beta_{i+\frac12}\) are the numerical face
coefficients. On intervals lying entirely to the left or right of the
interface, they coincide with \(\beta(x)\)
whereas on the cut interval \((x_I,x_{I+1})\),
\(\beta_{I+\frac12}\) denotes an effective positive interface
coefficient determined by the particular discretization.
This form includes standard pointwise coefficient approximations as well as interface constructions based on effective face conductivities.
Here \(r_i^h=f(x_i)\) at ordinary grid points.  At the two
interface-adjacent nodes, \(r_I^h\) and \(r_{I+1}^h\) may also contain
the correction terms arising from the prescribed jump conditions.
Thus \eqref{eq:generalsys} represents the matrix form of the
discretization after the interface corrections have been transferred
to the right-hand side.
Equivalently, before these corrections are transferred to the
right-hand side, the discrete equations may be written in conservative
flux-difference form
\begin{equation}\label{eq:flux_form}
\frac{1}{h}
\left(
F_{i-\frac12}^h-F_{i+\frac12}^h
\right)
=f(x_i),
\end{equation}
with
\begin{equation}\label{eq:flux-def}
F_{i+\frac12}^h
=
\beta_{i+\frac12}
\frac{U_{i+1}-U_i}{h}
\end{equation}
on ordinary grid intervals.
At an interval cut by the interface, the ordinary flux is replaced by
a corrected numerical flux incorporating the prescribed jump data.
When these correction terms are moved to the right-hand side, the
matrix retains the flux form in \eqref{eq:generalsys}, while the
corresponding right-hand-side entries at the interface-adjacent nodes
are modified.

We first derive an explicit formula for the inverse of the general
flux-form tridiagonal matrix in \eqref{eq:generalsys}. Let \(A_h\)
denote the tridiagonal matrix appearing on the left-hand side of
\eqref{eq:generalsys}, and define the scaled finite-difference operator
by
\[
L_h:=\frac{1}{h^2}A_h.
\]
Thus the discrete problem may be written as
\(
L_hU=r_h.
\)
The discrete Green kernel associated with \(L_h\) is defined by
\[
G_{ij}:=(L_h^{-1})_{ij},
\qquad 1\le i,j\le n.
\]The derivation
reveals a simple structural property: for each column of the discrete
Green's function, the weighted first differences have exactly two
constant plateaus, one on each side of the source node.  After scaling
to the finite-difference operator, the entries \(G_{ij}\) of its
inverse satisfy
\begin{equation}\label{eq:Goh-Goh2}
G_{ij}=O(h),
\qquad
G_{i,j+1}-G_{ij}=O(h^2),
\end{equation}
together with a mesh-independent maximum-norm stability estimate.

These properties explain how large interface truncation errors are
filtered by the inverse discrete operator.
For the error analysis,
let
\(
u_h:=(u(x_1),\ldots,u(x_n))^T
\)
denote the exact solution sampled at the interior grid points, and
define the truncation-error vector by
\[
\tau:=L_hu_h-r_h.
\]
Thus \(\tau_i\) is the residual obtained by inserting the exact grid
values into the \(i\)-th assembled discrete equation. In particular,
\(\tau_I\) and \(\tau_{I+1}\) denote the two components corresponding to
the grid points adjacent to the interface
\(x_I<\alpha<x_{I+1}\).  Suppose that
\begin{equation}\label{eq:O1Oh}
\tau_I=O(1),
\qquad
\tau_{I+1}=O(1),
\qquad
\tau_I+\tau_{I+1}=O(h).
\end{equation}
For the weighted harmonic interface discretization, these
properties are established in
Lemma~\ref{lem:harmonic-cancellation}.

Hence the contribution to the error can be written as
\[
G_{iI}(\tau_I+\tau_{I+1})
+
\bigl(G_{i,I+1}-G_{iI}\bigr)\tau_{I+1}.
\]
The Green's-function estimates then give
\[
G_{iI}(\tau_I+\tau_{I+1})=O(h^2),
\]
and
\[
\bigl(G_{i,I+1}-G_{iI}\bigr)\tau_{I+1}=O(h^2).
\]
Thus two individually \(O(1)\) local defects can together make only an
\(O(h^2)\) contribution to the global error.

This viewpoint shifts the emphasis from pointwise consistency to the
interaction between localized defects and the global structure of the
discrete operator.  In particular, the cancellation is closely tied
to conservation.  The two equations adjacent to the interface share
a common numerical cut-face flux, which occurs with opposite signs.
When the equations are added, this internal flux cancels exactly.
The remaining combined consistency error is then one order smaller
than the individual interface defects.

This observation leads naturally to a class of conservative interface
schemes.  We formulate an admissible family characterized by a common
cut-face flux, positive face coefficients, \(O(1)\) individual
interface defects, and the combined consistency condition
\[
\tau_I+\tau_{I+1}=O(h).
\]
For every scheme satisfying these structural properties, the discrete
Green's-function argument yields second-order global convergence in the
maximum norm.

The weighted harmonic construction is an important member of this
class.  If
\[
\alpha=x_I+\theta h,
\qquad 0<\theta<1,
\]
the effective coefficient on the cut interval is
\[
\widehat\beta_H
=
\left(
\frac{\theta}{\beta^-}
+
\frac{1-\theta}{\beta^+}
\right)^{-1}.
\]
Together with suitable corrections for the prescribed jumps, this
construction produces two interface-adjacent truncation errors that
are individually \(O(1)\) but whose leading terms cancel in their
sum. Harmonic-type interface discretizations have been studied from
the viewpoint of local flux approximations and stencil constructions;
see, for example, \cite{EwingIlievLazarovNaumovich2007,pan2026harmonic}.
The present analysis provides a complementary viewpoint based on the
discrete Green's function.  This shifts the emphasis from pointwise
consistency to the interaction between localized defects and the
global structure of the discrete operator: conservation controls the
relation between the neighboring interface defects, while the
discrete Green's function determines how those defects are propagated
into the global solution.

The main contributions of this paper can be summarized as follows:
\begin{enumerate}
\item
We derive an explicit discrete Green's function for a general
one-dimensional symmetric flux-form tridiagonal operator with
positive face coefficients and identify the two-plateau structure of
its weighted increments.

\item
We show how the estimates
\[
G_{ij}=O(h),
\qquad
G_{i,j+1}-G_{ij}=O(h^2),
\]
combine with the interface relation
\[
\tau_I+\tau_{I+1}=O(h)
\]
to convert individually \(O(1)\) local interface errors into an
\(O(h^2)\) global contribution.

\item
We formulate a family of admissible conservative interface schemes
for which this mechanism yields second-order maximum-norm
convergence.  The weighted harmonic discretization is included as a
particular member, and its cancellation property is established by
 a direct local consistency and flux-balance analysis.

\item
For Cartesian grids with a flat interface parallel to a coordinate direction, we show that
the conservative cancellation mechanism extends row by row or column
by column, and we identify the additional issues required for a full
multidimensional Green's-function analysis.
\end{enumerate}

The analysis therefore separates three ingredients in the convergence
mechanism: conservative flux balance, cancellation of neighboring
interface defects, and propagation of those defects through the
discrete Green's function.  Second-order pointwise consistency at every
grid point is not required.

The remainder of the paper is organized as follows.
Section~\ref{sec:Green} derives the discrete Green's function and its
structural estimates.  Section~\ref{sec:stability} discusses stability
and conditioning of the flux-form operator.
Section~\ref{sec:interface-error} analyzes the large local interface
errors and their cancellation.
Section~\ref{sec:family} develops the family of admissible
conservative interface schemes.
Section~\ref{sec:2d} discusses the extension of the conservative
cancellation mechanism to Cartesian grids in two dimensions.
Section~\ref{sec:numerics} presents one- and two-dimensional numerical
experiments that directly illustrate the cancellation mechanism, showing that the
individual interface defects remain $O(1)$ while their combined defect
is $O(h)$ and the global error exhibits second-order convergence,
including under changing relative interface positions.
The detailed local consistency and flux-balance calculation for the
weighted harmonic member is given in Appendix~A.

\section{Discrete Green's Function and Error Representation}
\label{sec:Green}

We begin with the one-dimensional flux-form difference operator that
underlies the interface schemes considered later. Throughout this section, we assume that
\[
\beta_{i+\frac12}>0,
\qquad i=0,\ldots,n.
\]
For a grid function \(U=(U_1,\ldots,U_n)^T\), with
\(U_0=U_{n+1}=0\), the matrix \(A_h\) introduced in the preceding
section acts as
\[
(A_hU)_i
=
-\beta_{i-\frac12}U_{i-1}
+
\left(\beta_{i-\frac12}+\beta_{i+\frac12}\right)U_i
-
\beta_{i+\frac12}U_{i+1},
\qquad 1\le i\le n.
\]
Thus \(A_h\) is the unscaled stiffness matrix associated with the
positive operator \(-(\beta u_x)_x\).  The standard finite-difference
operator is
\[
L_h:=\frac1{h^2}A_h.
\]

The distinction between \(A_h\) and \(L_h\) will be maintained
throughout.  The inverse of \(A_h\) is useful algebraically, whereas
the discrete Green's function entering the truncation-error analysis is
the inverse of \(L_h\).

For later use, introduce the reciprocal conductivities
\begin{equation}\label{eq:reciprocal}
s_k:=\frac1{\beta_{k-\frac12}},
\qquad k=1,\ldots,n+1,
\end{equation}
and the partial sums
\begin{equation}\label{eq:sum}
S(c,d):=
\begin{cases}
\displaystyle\sum_{k=c}^{d}s_k, & c\le d,\\[2mm]
0, & c>d.
\end{cases}
\end{equation}

\subsection{Explicit inverse of the flux-form matrix}

The inverse of \(A_h\) has a particularly simple form.

\begin{theorem}
\label{thm:Ah-inverse}
For \(1\le i,j\le n\),
\[
(A_h^{-1})_{ij}
=
\frac{
S(1,\min\{i,j\})\,
S(\max\{i,j\}+1,n+1)
}{
S(1,n+1)
}.
\]
In particular,
\[
(A_h^{-1})_{ij}>0,
\qquad
(A_h^{-1})_{ij}=(A_h^{-1})_{ji}.
\]
\end{theorem}

\begin{proof}
Fix \(j\), and let
\[
U=A_h^{-1}e_j,
\]
where \(e_j\) is the \(j\)-th coordinate vector.  Introduce the
discrete flux
\[
q_i
:=
\beta_{i-\frac12}(U_i-U_{i-1}),
\qquad
i=1,\ldots,n+1.
\]
The equation \(A_hU=e_j\) is equivalent to
\[
q_i-q_{i+1}=\delta_{ij},
\qquad i=1,\ldots,n.
\]
Hence there is a constant \(c\) such that
\[
q_i=
\begin{cases}
c, & 1\le i\le j,\\
c-1, & j+1\le i\le n+1.
\end{cases}
\]
Since \(U_0=U_{n+1}=0\),
\[
0
=
\sum_{k=1}^{n+1}(U_k-U_{k-1})
=
\sum_{k=1}^{n+1}s_kq_k.
\]
Therefore
\[
c\,S(1,j)+(c-1)S(j+1,n+1)=0,
\]
and consequently
\[
c=\frac{S(j+1,n+1)}{S(1,n+1)}.
\]

If \(i\le j\), then
\[
U_i
=
\sum_{k=1}^{i}s_kq_k
=
\frac{S(1,i)S(j+1,n+1)}
     {S(1,n+1)}.
\]
If \(i>j\), then
\[
U_i
=
cS(1,j)+(c-1)S(j+1,i),
\]
which, after using
\[
S(1,n+1)
=
S(1,j)+S(j+1,i)+S(i+1,n+1),
\]
gives
\[
U_i
=
\frac{S(1,j)S(i+1,n+1)}
     {S(1,n+1)}.
\]
Combining the two cases proves the stated formula.  Positivity and
symmetry are immediate.
\end{proof}

\subsection{The scaled discrete Green's function}

The finite-difference approximation to
\(-(\beta u_x)_x\) is \(L_h=h^{-2}A_h\).  We therefore define the
discrete Green's kernel by
\begin{equation}\label{eq:Green}
G_{ij}:=(L_h^{-1})_{ij}.
\end{equation}
Since
\[
L_h^{-1}=h^2A_h^{-1},
\]
Theorem~\ref{thm:Ah-inverse} immediately gives the following formula.

\begin{proposition}
\label{prop:Green-kernel}
For \(1\le i,j\le n\),
\[
G_{ij}
=
\frac{h^2}{S(1,n+1)}
\begin{cases}
S(1,i)\,S(j+1,n+1), & i\le j,\\[1mm]
S(1,j)\,S(i+1,n+1), & i>j.
\end{cases}
\]
In particular,
\begin{equation}\label{eq:sym_positive}
G_{ij}>0,
\qquad
G_{ij}=G_{ji}.
\end{equation}

If
\[
\Delta_iG^{(j)}:=G_{i+1,j}-G_{ij},
\]
then
\begin{equation}\label{eq:plateau}
\beta_{i+\frac12}\Delta_iG^{(j)}
=
\frac{h^2}{S(1,n+1)}
\begin{cases}
S(j+1,n+1), & i<j,\\[1mm]
-S(1,j), & i\ge j.
\end{cases}
\end{equation}
Consequently, for $1<i<n$
\[
\beta_{i-\frac12}\Delta_{i-1}G^{(j)}
-
\beta_{i+\frac12}\Delta_iG^{(j)}
=
h^2\delta_{ij}.
\]
Thus, for every fixed source index \(j\), the weighted Green's
increments
\[
\beta_{i+\frac12}\Delta_i G^{(j)}
\]
take exactly two constant values: one for \(i<j\) and the other for
\(i\ge j\). We refer to this as the two-plateau structure of the
weighted Green's increments.  This property is independent of any interface location or particular
averaging rule; it follows solely from the one-dimensional
tridiagonal structure of the conservative operator.
\end{proposition}

\begin{proof}
The formula for \(G_{ij}\) follows directly from
\[
L_h^{-1}=h^2A_h^{-1}
\]
and Theorem~\ref{thm:Ah-inverse}.  For \(i<j\),
\[
\Delta_iG^{(j)}
=
\frac{h^2}{S(1,n+1)}
\bigl(S(1,i+1)-S(1,i)\bigr)
S(j+1,n+1),
\]
and hence
\[
\beta_{i+\frac12}\Delta_iG^{(j)}
=
\frac{h^2S(j+1,n+1)}
     {S(1,n+1)}.
\]
Similarly, if \(i\ge j\),
\[
\beta_{i+\frac12}\Delta_iG^{(j)}
=
-\frac{h^2S(1,j)}
       {S(1,n+1)}.
\]
The last identity follows from the jump between these two constant
values at \(i=j\).
\end{proof}

\subsection{Bounds for the discrete Green's kernel}

\begin{corollary}
\label{cor:Green-bounds}
Let
\[
\beta_{\min}
:=
\min_{0\le i\le n}\beta_{i+\frac12}>0.
\]
There exists a constant \(C\), depending only on \(b-a\) and
\(\beta_{\min}\), such that
\[
0<G_{ij}\le Ch,
\qquad
1\le i,j\le n,
\]
and
\[
|G_{i+1,j}-G_{ij}|
\le
\frac{h^2}{\beta_{\min}}.
\]

By the symmetry \(G_{ij}=G_{ji}\), the same estimate also holds
for increments in the second index:
\[
|G_{i,j+1}-G_{ij}|
\le
\frac{h^2}{\beta_{\min}}.
\]

Moreover,
\[
\|L_h^{-1}\|_{\infty}
\le
\frac{(b-a)^2}{4\beta_{\min}}.
\]
\end{corollary}

\begin{proof}
From Proposition~\ref{prop:Green-kernel},
\[
G_{ij}
=
h^2
\frac{XY}{S(1,n+1)},
\]
where
\[
X=S(1,\min\{i,j\}),
\qquad
Y=S(\max\{i,j\}+1,n+1).
\]
Since \(S(1,n+1)\ge X+Y\),
\[
\frac{XY}{S(1,n+1)}
\le
\frac{XY}{X+Y}
\le
\min\{X,Y\}.
\]
Because each \(s_k\le\beta_{\min}^{-1}\),
\[
G_{ij}\le
\frac{(n+1)h^2}{\beta_{\min}}
=
\frac{(b-a)h}{\beta_{\min}}.
\]
This gives the first estimate.

The increment formula in
Proposition~\ref{prop:Green-kernel} gives
\[
|\Delta_iG^{(j)}|
\le
\frac{h^2}{\beta_{i+\frac12}}
\le
\frac{h^2}{\beta_{\min}}.
\]

For the matrix norm, using
\[
\frac{XY}{X+Y}\le \frac{X+Y}{4},
\]
one obtains
\[
(A_h^{-1})_{ij}
\le
\frac{
\min\{i,j\}+n+1-\max\{i,j\}
}{
4\beta_{\min}
}
\le
\frac{n+1}{4\beta_{\min}}.
\]
Hence
\[
\|A_h^{-1}\|_\infty
\le
\frac{n(n+1)}{4\beta_{\min}}.
\]
Multiplication by \(h^2\), together with
\(h=(b-a)/(n+1)\), gives
\[
\|L_h^{-1}\|_\infty
=
h^2\|A_h^{-1}\|_\infty
\le
\frac{(b-a)^2}{4\beta_{\min}}.
\]
\end{proof}

The particular constant in the last estimate is not essential for
the subsequent analysis.  What will matter is that
\[
\|L_h^{-1}\|_\infty=O(1),\qquad
G_{ij}=O(h),\qquad
G_{i+1,j}-G_{ij}=O(h^2).
\]

\subsection{Error representation and cancellation structure}

Let \(u_h=(u(x_1),...,u(x_n))^T\) denote the exact solution sampled at the grid
points, and let \(U=(U_1,...,U_n)^T\) be the numerical solution.  Following the
notation introduced in \eqref{eq:generalsys}, let
\begin{equation}\label{eq:discrete-rhs}
r_h=(r_1^h,\ldots,r_n^h)^T
\end{equation}
denote the assembled discrete right-hand side, so that
\[
L_hU=r_h.
\]

Define
\begin{equation}\label{eq:global-error}
e_i:=U_i-u_i,\quad u_i=u(x_i).
\end{equation}
The local truncation error is
\begin{equation}\label{eq:lte}
\tau:=L_hu_h-r_h,
\end{equation}
and hence
\[
L_he
=
L_h(U-u_h)
=
r_h-L_hu_h
=
-\tau.
\]
Equivalently,
\[
e_i
=
-\sum_{j=1}^{n}G_{ij}\tau_j.
\]
The overall sign is immaterial for the estimates below; what is
important is the way in which localized components of \(\tau\) are
weighted by the Green's kernel.

Suppose, in particular, that two neighboring truncation errors
\(\tau_I\) and \(\tau_{I+1}\) satisfy
\[
\tau_I=O(1),
\qquad
\tau_{I+1}=O(1),
\qquad
\tau_I+\tau_{I+1}=O(h).
\]
Their contribution to the error at node \(i\) can be written as
\[
G_{iI}\tau_I+G_{i,I+1}\tau_{I+1}
=
G_{iI}(\tau_I+\tau_{I+1})
+
\bigl(G_{i,I+1}-G_{iI}\bigr)\tau_{I+1}.
\]
By Corollary~\ref{cor:Green-bounds},
\[
G_{iI}=O(h),
\qquad
G_{i,I+1}-G_{iI}=O(h^2).
\]
Therefore
\[
G_{iI}\tau_I+G_{i,I+1}\tau_{I+1}
=
O(h)O(h)+O(h^2)O(1)
=
O(h^2).
\]

This identity isolates the Green's-function mechanism used throughout
the remainder of the paper: the small variation of the Green's kernel
between neighboring source nodes converts cancellation of the two
interface defects into a second-order contribution to the global
error. It explains how two
interface-adjacent truncation errors may each be \(O(1)\), while
their combined effect on the global solution remains \(O(h^2)\).
The essential ingredients are the conservative flux structure of
the discrete operator, the two-plateau behavior of its Green's
increments, and the cancellation of the leading interface defects.

\section{Stability and Conditioning}
\label{sec:stability}

The explicit discrete Green's function obtained in
Section~\ref{sec:Green} gives an immediate maximum-norm stability
estimate for the scaled difference operator
\[
L_h=\frac1{h^2}A_h.
\]
Indeed, Corollary~\ref{cor:Green-bounds} yields
\[
\|L_h^{-1}\|_\infty
\le
\frac{(b-a)^2}{4\beta_{\min}},
\qquad
\beta_{\min}
=
\min_{0\le i\le n}\beta_{i+\frac12}.
\]
Thus the discrete solution of
\[
L_hU=F
\]
satisfies
\[
\|U\|_\infty
\le
C\|F\|_\infty,
\]
where \(C\) is independent of the mesh size.

This distinction between \(A_h\) and \(L_h=h^{-2}A_h\) is important.
The inverse of the unscaled stiffness matrix \(A_h\) is not uniformly
bounded as \(h\to0\); rather,
\[
A_h^{-1}=h^{-2}L_h^{-1}.
\]
The mesh-independent stability estimate belongs to the scaled
finite-difference operator \(L_h\).

We next record how these stability and conditioning properties depend
on the coefficient contrast. Suppose that
\[
0<\beta_{\min}\le \beta_{i+\frac12}\le\beta_{\max}.
\]
The maximum-norm stability estimate depends only on the lower
ellipticity bound:
\[
\|L_h^{-1}\|_\infty
\le
\frac{(b-a)^2}{4\beta_{\min}}.
\]
In particular, for a fixed positive \(\beta_{\min}\), increasing some
of the conductivities does not destroy the uniform stability estimate.

The spectral condition number, however, depends on both the mesh size
and the coefficient contrast.  Since
\[
v^TA_hv
=
\sum_{i=0}^{n}
\beta_{i+\frac12}(v_{i+1}-v_i)^2,
\qquad
v_0=v_{n+1}=0,
\]
comparison with the constant-coefficient discrete Laplacian gives
\[
\beta_{\min}\lambda_k(T_h)
\le
\lambda_k(A_h)
\le
\beta_{\max}\lambda_k(T_h),
\]
where
\[
T_h=
\begin{pmatrix}
2&-1\\
-1&2&-1\\
&\ddots&\ddots&\ddots\\
&&-1&2
\end{pmatrix}.
\]
Consequently,
\[
\kappa_2(A_h)
\le
\frac{\beta_{\max}}{\beta_{\min}}
\kappa_2(T_h).
\]
Because multiplication by \(h^{-2}\) does not change a spectral
condition number,
\[
\kappa_2(L_h)=\kappa_2(A_h).
\]
Using
\[
\lambda_k(T_h)
=
4\sin^2\left(\frac{k\pi}{2(n+1)}\right),
\]
we obtain
\[
\kappa_2(L_h)
=
O\left(
\frac{\beta_{\max}}{\beta_{\min}}\,h^{-2}
\right).
\]

Thus two different notions should be distinguished.  The inverse
operator \(L_h^{-1}\) is uniformly bounded in the maximum norm under
a fixed lower ellipticity bound, while the algebraic condition number
of the discrete linear system grows like \(h^{-2}\) and may also
deteriorate with the coefficient contrast.


For the interface schemes considered below, the cut-face coefficient
is assumed positive.  In particular, the weighted harmonic
coefficient has this property whenever the physical coefficients on
the two sides of the interface are positive.  The resulting matrix
therefore retains the symmetric positive definite flux form analyzed
above.

Accordingly, the large local truncation errors at the two nodes
adjacent to the interface are not a consequence of instability of
the discrete operator.  Their global effect is instead governed by
the Green's-function cancellation mechanism identified in
Section~\ref{sec:Green}.  We now verify that the interface
discretization has precisely the required truncation-error structure.

\section{Large Local Errors and Their Cancellation}
\label{sec:interface-error}

We now apply the discrete Green's-function framework developed in
Section~\ref{sec:Green} to interface discretizations with large local
errors near the interface.  The main point is that second-order global accuracy
does not require second-order local consistency at the two grid points
adjacent to the interface.  Instead, it is sufficient that the two
large local defects satisfy a suitable cancellation relation.

Let
\[
\alpha=x_I+\theta h,
\qquad 0<\theta<1,
\]
so that
\[
x_I<\alpha<x_{I+1}.
\]
Assume
\[
\beta(x)=
\begin{cases}
\beta^-, & x<\alpha,\\
\beta^+, & x>\alpha,
\end{cases}
\qquad
\beta^\pm>0,
\]
and let the interface conditions be
\[
[u]_\alpha=w,
\qquad
[\beta u_x]_\alpha=v.
\]

For the weighted harmonic interface scheme considered below, the
effective coefficient on the cut interval is
\begin{equation}\label{eq:hatbeta}
\widehat\beta_H
=
\left(
\frac{\theta}{\beta^-}
+
\frac{1-\theta}{\beta^+}
\right)^{-1},
\end{equation}
with the prescribed solution and flux jumps incorporated through the
corresponding correction terms.  A broader class of conservative
interface schemes will be developed independently in Section~\ref{sec:family}.

\subsection{Interface truncation errors}

Let \(L_h\) denote the scaled discrete operator introduced in
Section~\ref{sec:Green}.  We now examine the truncation errors
associated with the interface discretization.  At an ordinary grid
point \(x_i\), away from the interface, the discrete right-hand side
satisfies
\[
r_i^h=f(x_i),
\]
and the standard local consistency calculation gives
\[
\tau_i=O(h^2).
\]

For the weighted harmonic scheme, the precise assembled right-hand-side
terms corresponding to \(r_I^h\) and \(r_{I+1}^h\) are given in
Appendix~\ref{app:interface-cancellation}.

At the two grid points \(x_I\) and \(x_{I+1}\) adjacent to the
interface, the assembled right-hand side contains the jump corrections
associated with the cut-face flux.  The corresponding truncation
errors \(\tau_I\) and \(\tau_{I+1}\) require a separate analysis.
The relevant cancellation assumption is the following.

\begin{assumption}[Interface cancellation property]
\label{ass:interface-cancellation}
The interface corrections are chosen so that
\[
\tau_I=O(1),
\qquad
\tau_{I+1}=O(1),
\]
and
\[
\tau_I+\tau_{I+1}=O(h).
\]
\end{assumption}

Thus the two exceptional truncation errors may remain individually
large as \(h\to0\), but their leading terms cancel when the two
defects are combined.

For the weighted harmonic interface scheme, this property follows
from a direct one-sided Taylor expansion about the interface.

\begin{lemma}[Cancellation for the weighted harmonic scheme]
\label{lem:harmonic-cancellation}
Let the weighted harmonic coefficient be as defined in \eqref{eq:hatbeta}, and let
the jump corrections be those obtained by matching the prescribed
conditions
\[
[u]_\alpha=w,
\qquad
[\beta u_x]_\alpha=v.
\]
If
\[
u^-\in C^3([x_{I-1},\alpha]),
\qquad
u^+\in C^3([\alpha,x_{I+2}]),
\] then
\[
\tau_I=O(1),
\qquad
\tau_{I+1}=O(1),
\qquad
\tau_I+\tau_{I+1}=O(h).
\]
\end{lemma}

\begin{proof}
The detailed one-sided Taylor expansion and flux-balance calculation
is given in Appendix~\ref{app:interface-cancellation}.
\end{proof}

The preceding lemma verifies
Assumption~\ref{ass:interface-cancellation} for the weighted harmonic
scheme.  Other conservative interface schemes may be treated in the same
framework whenever they satisfy these three truncation-error
relations; a broader class will be formulated in
Section~\ref{sec:family}.

\subsection{Green's-function reduction of the interface error}

Let
\[
G_{ij}=(L_h^{-1})_{ij}
\]
be the discrete Green's kernel from Section~\ref{sec:Green}.  If
\[
e=U-u_h,
\]
then
\[
L_he=-\tau,
\]
and hence
\[
e_i
=
-\sum_{j=1}^{n}G_{ij}\tau_j.
\]

Separate the regular-grid and interface contributions:
\[
e_i
=
-\sum_{\substack{j=1\\j\ne I,I+1}}^n
G_{ij}\tau_j
-
G_{iI}\tau_I
-
G_{i,I+1}\tau_{I+1}.
\]

The two exceptional terms may be rearranged as
\begin{equation}\label{eq:central}
G_{iI}\tau_I+G_{i,I+1}\tau_{I+1}
=
G_{iI}(\tau_I+\tau_{I+1})
+
\bigl(G_{i,I+1}-G_{iI}\bigr)\tau_{I+1}.
\end{equation}
By Corollary~\ref{cor:Green-bounds},
\[
G_{iI}=O(h),
\qquad
G_{i,I+1}-G_{iI}=O(h^2).
\]
Therefore, under
Assumption~\ref{ass:interface-cancellation},
\[
G_{iI}(\tau_I+\tau_{I+1})
=
O(h)O(h)
=
O(h^2),
\]
and
\[
\bigl(G_{i,I+1}-G_{iI}\bigr)\tau_{I+1}
=
O(h^2)O(1)
=
O(h^2).
\]
Hence
\[
G_{iI}\tau_I+G_{i,I+1}\tau_{I+1}
=
O(h^2).
\]

The large local interface defects therefore make only a second-order
contribution to the global error.

\subsection{Second-order global convergence}

We can now state the convergence result under the cancellation
structure identified above.

\begin{theorem}[Second-order accuracy under interface cancellation]
\label{thm:cancellation-second-order}
Assume that the discrete operator has the conservative flux form of
Section~\ref{sec:Green}, that
\[
\tau_i=O(h^2),
\qquad
i\ne I,I+1,
\]
and that the two interface truncation errors satisfy
Assumption~\ref{ass:interface-cancellation}.  Then
\[
\|U-u_h\|_\infty
\le Ch^2,
\]
where \(C\) is independent of \(h\).
\end{theorem}

\begin{proof}
For the regular nodes,
\[
|\tau_i|\le Ch^2,
\qquad
i\ne I,I+1.
\]
Using the maximum-norm stability estimate
\[
\|L_h^{-1}\|_\infty\le C,
\]
their total contribution is \(O(h^2)\).

The contribution from the two interface nodes is \(O(h^2)\) by the
Green's-function decomposition above.  Combining the regular and
interface contributions proves the result.
\end{proof}

As an immediate consequence, the weighted harmonic interface scheme
satisfies the same second-order estimate.

\begin{corollary}
\label{cor:harmonic-second-order}
Under the regularity assumptions of
Lemma~\ref{lem:harmonic-cancellation}, the weighted harmonic interface
scheme satisfies
\[
\|U-u_h\|_\infty\le Ch^2.
\]
\end{corollary}

Lemma~\ref{lem:harmonic-cancellation} shows that the weighted harmonic
scheme has exactly the truncation-error structure required by the
preceding theorem.  The argument, however, suggests that the mechanism
is not specific to this particular choice of cut-face coefficient.
What is essential is the conservative coupling of the two
interface-adjacent equations together with the resulting combined
consistency condition.  This observation motivates the broader class
of interface schemes considered in the next section.

\section{A Class of Conservative Interface Schemes}
\label{sec:family}

We now formulate the broader class suggested by the preceding
analysis.  The essential structural requirement is that the two
equations adjacent to the interface share a single numerical flux on
the cut interval.  When the two equations are added, this common flux
cancels exactly.  When this conservation property is combined with \(O(1)\) local
consistency at the two interface-adjacent nodes and first-order
consistency of the combined equation, the interface defects have the
cancellation structure required by
Theorem~\ref{thm:cancellation-second-order}.

This observation leads naturally to a family of interface schemes.
Let
\[
x_I<\alpha<x_{I+1},
\qquad
\alpha=x_I+\theta h,
\qquad
0<\theta<1,
\]
and let
\[
[u]_\alpha=w,
\qquad
[\beta u_x]_\alpha=v.
\]
We consider a numerical flux on the cut interval of the form
\begin{equation}\label{eq:cut-flux}
F_{I+\frac12}^h
=
\frac{\gamma}{h}
\left(
U_{I+1}-U_I-c_w w
\right)
-c_v v,
\end{equation}
where
\[
\gamma>0,
\qquad
c_w,\;c_v\in\mathbb R
\]
may depend on
\[
\theta,\qquad \beta^-,\qquad \beta^+,
\]
but not on \(h\).

The two equations adjacent to the interface are written using this
same cut-face flux:
\[
\frac{1}{h}
\left(
F_{I-\frac12}^h-F_{I+\frac12}^h
\right)
=f(x_I),
\]
and
\[
\frac{1}{h}
\left(
F_{I+\frac12}^h-F_{I+\frac32}^h
\right)
=f(x_{I+1})
\]
with the ordinary fluxes
\[
F_{i+\frac12}^h
=
\beta_{i+\frac12}
\frac{U_{i+1}-U_i}{h}
\]
away from the interface.

Because exactly the same quantity \(F_{I+\frac12}^h\) occurs in the
two equations with opposite signs, the cut-face flux cancels
identically when the equations are added.  This conservation
property is independent of the particular choice of
\(\gamma,c_w,c_v\).

We next determine conditions on the parameters for which the cut-face
flux is consistent with the interface conditions.  For the exact
solution,
\[
u_{I+1}-u_I
=
w
+
h\left(
\theta u_x^-+(1-\theta)u_x^+
\right)
+O(h^2).
\]
Hence
\[
F_{I+\frac12}^h
=
\frac{\gamma}{h}(1-c_w)w
+
\gamma
\left(
\theta u_x^-+(1-\theta)u_x^+
\right)
-c_vv
+O(h).
\]
To avoid an \(O(h^{-1})\) contribution when \(w\ne0\), it is necessary
to choose
\[
c_w=1.
\]

Using
\[
\beta^+u_x^+-\beta^-u_x^-=v,
\]
we may write
\[
u_x^+
=
\frac{\beta^-}{\beta^+}u_x^-
+\frac{v}{\beta^+}.
\]
Therefore
\[
\theta u_x^-+(1-\theta)u_x^+
=
\left(
\theta+(1-\theta)\frac{\beta^-}{\beta^+}
\right)u_x^-
+
\frac{1-\theta}{\beta^+}v.
\]
Thus
\[
F_{I+\frac12}^h
=
\gamma
\left(
\theta+(1-\theta)\frac{\beta^-}{\beta^+}
\right)u_x^-
+
\left(
\frac{\gamma(1-\theta)}{\beta^+}-c_v
\right)v
+O(h).
\]

If the cut-face flux is required to approximate the left physical
flux \(\beta^-u_x^-\), consistency gives
\[
\gamma
\left(
\theta+(1-\theta)\frac{\beta^-}{\beta^+}
\right)
=
\beta^-.
\]
Consequently,
\[
\gamma
=
\frac{\beta^-\beta^+}
{\theta\beta^++(1-\theta)\beta^-}
=
\left(
\frac{\theta}{\beta^-}
+
\frac{1-\theta}{\beta^+}
\right)^{-1}.
\]
The coefficient of \(v\) must then vanish, giving
\[
c_v
=
\frac{\gamma(1-\theta)}{\beta^+}
=
\frac{(1-\theta)\beta^-}
{\theta\beta^++(1-\theta)\beta^-}.
\]

Thus, within this particular three-parameter representation, exact
first-order consistency determines the familiar weighted harmonic
coefficient and its associated jump correction.

The calculation nevertheless suggests a broader structural class.
The freedom does not arise by arbitrarily replacing the harmonic
coefficient while retaining the same correction terms.  Rather,
alternative interface discretizations may be admitted if their flux
representation and distribution of the prescribed jump data preserve
conservation and the required local and combined consistency
properties.

Motivated by this observation, we introduce the following class.

\begin{definition}[Admissible interface scheme]
\label{def:admissible}
A two-node interface discretization is called admissible if:
\begin{enumerate}
\item the two interface-adjacent equations share a single numerical
      cut-face flux, occurring with opposite signs;

\item the resulting discrete operator has positive face coefficients
      and the conservative flux form of Section~\ref{sec:Green};

\item at all ordinary nodes,
      \[
      \tau_i=O(h^2);
      \]

\item at the two interface-adjacent nodes,
      \[
      \tau_I=O(1),
      \qquad
      \tau_{I+1}=O(1);
      \]

\item the two interface defects satisfy the combined consistency
      condition
      \[
      \tau_I+\tau_{I+1}=O(h).
      \]
\end{enumerate}
\end{definition}

The first condition expresses conservation and is the structural
reason that the cut-face contribution disappears when the two
interface equations are added.  Conditions (3)--(5) give the required truncation-error structure,
while condition (2) ensures that the Green's-function estimates of
Section~\ref{sec:Green} apply.  Hence
Theorem~\ref{thm:cancellation-second-order} yields the following
immediate consequence.

\begin{corollary}
\label{cor:admissible-family}
Every admissible interface scheme in the sense of
Definition~\ref{def:admissible} satisfies
\[
\|U-u_h\|_\infty\le Ch^2.
\]
\end{corollary}

\begin{proof}
This follows directly from
Theorem~\ref{thm:cancellation-second-order} and the defining
properties of an admissible scheme.
\end{proof}

The weighted harmonic construction considered in
Section~\ref{sec:interface-error} corresponds to
\[
\gamma
=
\widehat\beta_H
=
\left(
\frac{\theta}{\beta^-}
+
\frac{1-\theta}{\beta^+}
\right)^{-1},
\qquad
c_w=1,
\]
and
\[
c_v
=
\frac{(1-\theta)\beta^-}
{\theta\beta^++(1-\theta)\beta^-}.
\]
By Lemma~\ref{lem:harmonic-cancellation}, it satisfies the interface
cancellation conditions and is therefore an admissible member of the
class.

Thus the weighted harmonic discretization is one realization of the
broader conservative cancellation principle formalized in
Definition~\ref{def:admissible}.  Other interface discretizations are
covered by the same analysis whenever their cut-face flux and jump
corrections satisfy these structural conditions; the argument does
not depend on a particular closed formula for the effective interface
coefficient.

\section{The Cancellation Mechanism in Two Dimensions}
\label{sec:2d}

We next examine how the conservative cancellation mechanism
extends to Cartesian discretizations in two dimensions.  The purpose
of this section is not to develop a complete multidimensional
interface theory, but to identify the part of the one-dimensional
argument that survives directly for a flat interface parallel to a
coordinate direction.

Consider
\[
-\nabla\cdot(\beta\nabla u)=f
\qquad\text{in }\Omega\setminus\Gamma,
\]
where \(\Omega\subset\mathbb R^2\) is rectangular and the interface
\(\Gamma\) separates \(\Omega\) into two subdomains
\(\Omega^-\) and \(\Omega^+\).  Assume
\[
\beta=
\begin{cases}
\beta^-,&\Omega^-,\\
\beta^+,&\Omega^+,
\end{cases}
\qquad
\beta^\pm>0,
\]
and impose
\[
[u]_\Gamma=w,
\qquad
[\beta\partial_nu]_\Gamma=v.
\]

We use a Cartesian grid with mesh widths \(h_x\) and \(h_y\).  Away
from the interface, the standard conservative five-point operator is
used:
\[
(L_hU)_{ij}
=
-\frac{1}{h_x^2}
\left[
\beta_{i+\frac12,j}(U_{i+1,j}-U_{ij})
-
\beta_{i-\frac12,j}(U_{ij}-U_{i-1,j})
\right]
\]
\[
\hspace{20mm}
-\frac{1}{h_y^2}
\left[
\beta_{i,j+\frac12}(U_{i,j+1}-U_{ij})
-
\beta_{i,j-\frac12}(U_{ij}-U_{i,j-1})
\right].
\]

Suppose first that the interface is a vertical straight line
\[
\Gamma=\{x=\alpha\},
\]
with
\[
x_I<\alpha<x_{I+1}.
\]
Write
\[
\alpha=x_I+\theta h_x,
\qquad
0<\theta<1.
\]
Every horizontal grid line therefore encounters the same
one-dimensional cut geometry.

On each horizontal edge cut by the interface
\[
[(x_I,y_j),(x_{I+1},y_j)],
\]
we introduce the same type of conservative numerical flux used in
Section~\ref{sec:family}.  For the weighted harmonic member,
\[
\widehat\beta_{H,x}
=
\left(
\frac{\theta}{\beta^-}
+
\frac{1-\theta}{\beta^+}
\right)^{-1}.
\]
The jump data are incorporated through the corresponding correction
terms in the normal direction.  Since the interface is vertical, the
normal direction is the \(x\)-direction, whereas the ordinary
\(y\)-fluxes remain unchanged.

Thus, for every fixed \(j\), the pair of equations adjacent to the
interface contains one common cut-edge flux,
\[
F^h_{I+\frac12,j},
\]
with opposite signs.  Adding the two neighboring equations eliminates
this flux exactly.

This is the direct two-dimensional analogue of the conservative
cancellation identified in Sections~\ref{sec:interface-error}
and~\ref{sec:family}.

\begin{proposition}[Row-wise interface cancellation]
\label{prop:2d-cancellation}
Assume that the interface is the vertical line \(x=\alpha\), that
on every horizontal grid line the interface discretization is
obtained from an admissible one-dimensional conservative scheme of
Section~\ref{sec:family}, and that the one-sided derivatives required
in the local consistency estimates are uniformly bounded along the
interface.  Then, uniformly in \(j\),
\[
\tau_{I,j}=O(1),
\qquad
\tau_{I+1,j}=O(1),
\]
and
\[
\tau_{I,j}+\tau_{I+1,j}
=
O(h_x)+O(h_y^2).
\]
In particular, on a quasi-uniform Cartesian mesh,
\[
\tau_{I,j}+\tau_{I+1,j}=O(h).
\]
\end{proposition}

\begin{proof}
Write the local truncation error as the sum of its contributions in
the two coordinate directions,
\[
\tau_{i,j}=\tau^x_{i,j}+\tau^y_{i,j}.
\]
For each fixed \(j\), the \(x\)-direction discretization at the two
nodes adjacent to the vertical interface is an admissible
one-dimensional conservative interface discretization with mesh size
\(h_x\).  Hence the one-dimensional cancellation property gives
\[
\tau^x_{I,j}=O(1),
\qquad
\tau^x_{I+1,j}=O(1),
\qquad
\tau^x_{I,j}+\tau^x_{I+1,j}=O(h_x).
\]

In the \(y\)-direction no grid line crosses the interface.  The
standard centered discretization is therefore applied entirely within
a single smooth subdomain, and the usual Taylor expansion gives
\[
\tau^y_{I,j}=O(h_y^2),
\qquad
\tau^y_{I+1,j}=O(h_y^2).
\]
Consequently,
\[
\begin{aligned}
\tau_{I,j}+\tau_{I+1,j}
&=
\bigl(\tau^x_{I,j}+\tau^x_{I+1,j}\bigr)
+
\bigl(\tau^y_{I,j}+\tau^y_{I+1,j}\bigr)\\
&=O(h_x)+O(h_y^2).
\end{aligned}
\]
The individual estimates follow in the same way.  If \(h_x\) and
\(h_y\) are comparable, the asserted \(O(h)\) estimate follows.
\end{proof}
The same argument applies, with \(x\) and \(y\) interchanged, to a
horizontal interface.

The proposition identifies the part of the one-dimensional mechanism
that carries over directly to two dimensions.  Along every grid line
crossing the interface, the two interface-adjacent equations share one
cut-edge flux with opposite signs, so that this internal flux cancels
exactly when the equations are combined.  The jump corrections enter
through the right-hand side, while the positive face coefficients
retain the usual diffusion-type sign structure of the discrete
operator.

What does not carry over directly is the Green's-function step.  In
one dimension, the explicit representation in
Section~\ref{sec:Green} yields
\[
G_{ij}=O(h),
\qquad
G_{i,j+1}-G_{ij}=O(h^2),
\]
which converts the paired interface defects into an \(O(h^2)\) global
contribution.  The two-dimensional discrete Green's kernel has a
different structure, and corresponding multidimensional estimates
would be required for a complete maximum-norm convergence analysis.
We do not pursue that analysis here.

For a curved interface, additional geometric issues arise.  The
intersection location varies from one grid line to another, the
interface normal need not coincide with a coordinate direction, and
the prescribed normal-flux jump must be distributed among Cartesian
face fluxes in a geometrically consistent manner.

The preceding analysis nevertheless suggests a natural design
principle: interface corrections should be introduced through shared
conservative face fluxes whenever possible, thereby preserving exact
cancellation of internal numerical fluxes when neighboring equations
are combined.  A rigorous convergence theory for general curved
interfaces would additionally require geometric consistency estimates
and appropriate multidimensional discrete Green's-function bounds.
These questions lie beyond the scope of the present work.

\section{Numerical illustration of the cancellation mechanism}
\label{sec:numerics}

We conclude with one- and two-dimensional numerical experiments
illustrating the cancellation mechanism developed in the preceding
sections.  The
purpose of these experiments is not to provide an extensive numerical
comparison, but rather to display directly the distinction between
the large individual truncation errors near the interface and their
much smaller combined effect.

\subsection{One-dimensional example}

We first consider the interface problem on $(0,1)$ with
\[
\alpha=\frac13,
\qquad
\beta^- =1.5,
\qquad
\beta^+=3,
\]
and exact solution
\[
u(x)=
\begin{cases}
\sin(5x), & 0\le x<\alpha,\\[1mm]
\cos(3x)-\cos(3), & \alpha<x\le1.
\end{cases}
\]
Thus \(u(0)=u(1)=0\).
The source term and interface jumps are obtained from
the exact solution so that
\[
-(\beta u_x)_x=f,
\qquad
[u]_\alpha=w,
\qquad
[\beta u_x]_\alpha=v.
\]
A closely related test problem is used in \cite{pan2026harmonic}.  Here, however, our interest is specifically
in the behavior of the two interface truncation errors.

Let $x_I<\alpha<x_{I+1}$, and let $\tau_I$ and $\tau_{I+1}$ denote
the truncation errors in the symmetrically scaled interface equations
used in the analysis of Section~\ref{sec:interface-error}.  Table~\ref{tab:1d-cancellation} reports the two individual defects,
the normalized combined defect
\[
h^{-1}|\tau_I+\tau_{I+1}|,
\]
and the maximum-norm error.

\begin{table}[htbp]
\centering
\caption{One-dimensional interface truncation errors and global
maximum-norm error.}
\label{tab:1d-cancellation}
\begin{tabular}{c|cccc}
\hline
$N$
& $|\tau_I|$
& $|\tau_{I+1}|$
& $h^{-1}|\tau_I+\tau_{I+1}|$
& $\|U-u_h\|_\infty$
\\
\hline
32
& $9.5806$
& $9.2530$
& $10.482$
& $4.3434\times10^{-3}$
\\
64
& $0.7096$
& $0.7789$
& $4.437$
& $2.3005\times10^{-4}$
\\
128
& $9.4884$
& $9.4177$
& $9.051$
& $2.8438\times10^{-4}$
\\
256
& $0.6832$
& $0.7045$
& $5.444$
& $1.4146\times10^{-5}$
\\
512
& $9.4725$
& $9.4555$
& $8.691$
& $1.7971\times10^{-5}$
\\
1024
& $0.6800$
& $0.6856$
& $5.697$
& $8.8068\times10^{-7}$
\\
\hline
\end{tabular}
\end{table}

The individual interface defects clearly do not tend to zero; they
remain of order $O(1)$.  Their magnitudes alternate between two
families because the relative position of the fixed interface
$\alpha=1/3$ within the cut mesh interval changes under dyadic
refinement.  In contrast,
\[
h^{-1}|\tau_I+\tau_{I+1}|
\]
remains bounded.  Thus the computation directly exhibits the
cancellation relation
\[
\tau_I+\tau_{I+1}=O(h)
\]
used in the analysis.

The maximum-norm error displays the corresponding mesh-position
oscillation.  Comparing meshes for which the interface has the same
relative position gives
\[
4.3434\times10^{-3}
\longrightarrow
2.8438\times10^{-4}
\longrightarrow
1.7971\times10^{-5}
\]
for \(N=32,128,512\), and
\[
2.3005\times10^{-4}
\longrightarrow
1.4146\times10^{-5}
\longrightarrow
8.8068\times10^{-7}
\]
for \(N=64,256,1024\).
Since \(h\) is reduced by a factor of four along each subsequence,
these reductions are consistent with
\[
\|U-u_h\|_\infty=O(h^2).
\]
Thus the computation confirms the one-dimensional convergence result
and illustrates that the large individual local defects do not
determine the global convergence rate.
\medskip

To examine whether the observed convergence depends on the repeating
relative interface positions associated with $\alpha=1/3$, we repeat
the same experiment with
\[
        \alpha=\frac{1}{\sqrt{2}}.
\]
For each mesh, let
\[
        \theta_N=\frac{\alpha-x_I}{h},
        \qquad 0<\theta_N<1,
\]
denote the relative position of the interface in the cut interval
$(x_I,x_{I+1})$.  Since $\alpha$ is irrational, these relative
positions do not repeat under dyadic refinement.  Table~\ref{tab:1d-irrational}
reports $\theta_N$ together with the interface defects and the
normalized global error.

\begin{table}[htbp]
\centering
\caption{One-dimensional cancellation for $\alpha=1/\sqrt{2}$ under
changing relative interface positions.}
\label{tab:1d-irrational}
\begin{tabular}{c|ccccc}
\hline
$N$
& $\theta_N$
& $|\tau_I|$
& $|\tau_{I+1}|$
& $h^{-1}|\tau_I+\tau_{I+1}|$
& $h^{-2}\|U-u_h\|_\infty$
\\
\hline
32   & 0.62742  & 2.4443 & 2.7900 & 11.061 & 2.7985 \\
64   & 0.25483  & 2.0815 & 2.5718 & 31.374 & 2.0765 \\
128  & 0.50967  & 1.3568 & 1.3297 &  3.4678 & 2.4689 \\
256  & 0.019336 & 6.5864 & 6.7837 & 50.521 & 3.5722 \\
512  & 0.038672 & 6.2260 & 6.3219 & 49.130 & 3.4682 \\
1024 & 0.077344 & 5.4811 & 5.5262 & 46.227 & 3.2126 \\
2048 & 0.15469  & 4.0620 & 4.0815 & 39.954 & 2.6413 \\
4096 & 0.30938  & 1.5164 & 1.5226 & 25.537 & 2.1454 \\
\hline
\end{tabular}
\end{table}

The variation of $\theta_N$ is reflected in the nonmonotone behavior
of the individual interface defects and of their normalized sum.
Nevertheless, the individual defects remain of order $O(1)$ over the
refinements shown, while
\[
        h^{-1}|\tau_I+\tau_{I+1}|
\]
remains bounded.  Thus the computation is consistent with the
cancellation relation
\[
        \tau_I+\tau_{I+1}=O(h)
\]
without restricting attention to meshes having the same relative
interface position.  More importantly, the normalized global error
$h^{-2}\|U-u_h\|_\infty$ remains bounded despite the changing cut
configuration.  This provides numerical evidence for the
second-order maximum-norm convergence predicted by the analysis when
the relative interface position varies under refinement.

\subsection{Two-dimensional example}

We next consider the unit square
\[
\Omega=(0,1)^2
\]
with the vertical interface
\[
\Gamma=\{(x,y):x=\alpha\},
\qquad
\alpha=\frac13,
\]
and the same piecewise constant coefficient
\[
\beta^- =1.5,
\qquad
\beta^+=3.
\]
The exact solution is
\[
u(x,y)=
\begin{cases}
\sin(5x)\sin(\pi y), & x<\alpha,\\[1mm]
(\cos(3x)-\cos(3))\sin(\pi y), & x>\alpha.
\end{cases}
\]
Thus \(u=0\) on \(\partial\Omega\).  The source term and interface
jump data are determined from the exact solution.  This test is a
homogeneous-boundary modification of the two-dimensional example
considered in \cite{pan2026harmonic}.

The interface is parallel to a coordinate direction, as in
Section~\ref{sec:2d}. For each fixed horizontal grid line $y=y_j$,
let $\tau_{I,j}$ and $\tau_{I+1,j}$ denote the truncation errors at
the two nodes adjacent to the interface.  To examine the row-wise
cancellation uniformly in $y$, we report
\[
\max_j|\tau_{I,j}|,
\qquad
\max_j|\tau_{I+1,j}|,
\qquad
\frac{1}{h}
\max_j|\tau_{I,j}+\tau_{I+1,j}|.
\]

\begin{table}[htbp]
\centering
\caption{Two-dimensional row-wise interface cancellation and global
maximum-norm error.}
\label{tab:2d-cancellation}
\begin{tabular}{c|cccc}
\hline
$N$
& $\max_j|\tau_{I,j}|$
& $\max_j|\tau_{I+1,j}|$
& $h^{-1}\max_j|\tau_{I,j}+\tau_{I+1,j}|$
& $\|U-u_h\|_\infty$
\\
\hline
16  & $1.0727$ & $0.9067$ & $2.6554$ & $4.5895\times10^{-3}$\\
32  & $9.5944$ & $9.2232$ & $11.880$ & $4.4728\times10^{-3}$\\
64  & $0.7120$ & $0.7685$ & $3.6115$ & $2.7165\times10^{-4}$\\
128 & $9.4893$ & $9.4158$ & $9.4035$ & $2.9456\times10^{-4}$\\
256 & $0.6834$ & $0.7038$ & $5.2354$ & $1.6779\times10^{-5}$\\
512 & $9.4725$ & $9.4554$ & $8.7788$ & $1.8640\times10^{-5}$\\
\hline
\end{tabular}
\end{table}

The individual interface defects remain bounded but exhibit a marked
alternation with the relative position of the interface in the mesh.
In contrast,
\[
h^{-1}\max_j|\tau_{I,j}+\tau_{I+1,j}|
\]
remains bounded over the refinements shown, in agreement with the
row-wise cancellation estimate of
Proposition~\ref{prop:2d-cancellation}.

The maximum-norm errors exhibit the same alternating pattern.
Comparing meshes for which the interface has the same relative
position gives
\[
4.5895\times10^{-3}
\longrightarrow
2.7165\times10^{-4}
\longrightarrow
1.6779\times10^{-5}
\]
for \(N=16,64,256\), and
\[
4.4728\times10^{-3}
\longrightarrow
2.9456\times10^{-4}
\longrightarrow
1.8640\times10^{-5}
\]
for \(N=32,128,512\).
Since \(h\) is reduced by a factor of four along each subsequence,
these reductions are consistent with second-order convergence.

Thus the computation directly illustrates the row-wise cancellation
proved in Section~\ref{sec:2d}.  It also provides numerical evidence
that, in this flat-interface setting, the resulting global
maximum-norm error is second order.  A complete two-dimensional
Green's-function proof of the latter observation is beyond the scope
of the present analysis.

\appendix
\section{Interface Truncation-Error Cancellation}
\label{app:interface-cancellation}

In this appendix we prove
Lemma~\ref{lem:harmonic-cancellation}.  We use throughout the sign
convention of the present paper,
\[
-(\beta u_x)_x=f.
\]

Let
\[
x_I<\alpha<x_{I+1},
\qquad
\delta_\ell:=\alpha-x_I,
\qquad
\delta_r:=x_{I+1}-\alpha,
\]
so that
\[
\delta_\ell+\delta_r=h.
\]
For the symmetrically scaled interface equations, set
\[
h_\ell
:=
\alpha-x_{I-\frac12}
=
\delta_\ell+\frac h2,
\qquad
h_r
:=
x_{I+\frac32}-\alpha
=
\delta_r+\frac h2.
\]
Thus
\[
h_\ell+h_r=2h.
\]

Write
\[
w=[u]_\alpha,
\qquad
v=[\beta u_x]_\alpha,
\]
and define the weighted harmonic coefficient
\[
\widehat\beta_H
=
\left(
\frac{\delta_\ell}{\beta^-h}
+
\frac{\delta_r}{\beta^+h}
\right)^{-1}.
\]

Since \(\delta_\ell=\theta h\) and
\(\delta_r=(1-\theta)h\), this agrees with the
coefficient \(\widehat\beta_H\) defined in
Lemma~\ref{lem:harmonic-cancellation}.

The correction terms corresponding to the sign convention
\(-(\beta u_x)_x=f\) are
\[
C_I
=
-\frac{\widehat\beta_H}{h_\ell h}
\left(
w+\frac{v}{\beta^+}\delta_r
\right),
\]
and
\[
C_{I+1}
=
\frac{\widehat\beta_H}{h_r h}
\left(
w-\frac{v}{\beta^-}\delta_\ell
\right).
\]

These correction terms are the contributions transferred to the
right-hand side when the corrected cut-interval flux is written in
the symmetric matrix form.

Thus the two assembled right-hand-side entries in the symmetric
formulation are
\[
r_I^h
=
\frac{h_\ell}{h}(f_I+C_I),
\qquad
r_{I+1}^h
=
\frac{h_r}{h}(f_{I+1}+C_{I+1}),
\]
where
\[
f_I=f(x_I),
\qquad
f_{I+1}=f(x_{I+1}).
\]

Let
\[
F_{i+\frac12}
=
\beta_{i+\frac12}
\frac{u_{i+1}-u_i}{h}
\]
denote the ordinary numerical flux, and let
\[
\widehat F_{I+\frac12}
=
\widehat\beta_H
\frac{u_{I+1}-u_I}{h}
\]
denote the numerical flux on the cut interval.  The two interface
truncation errors are therefore
\[
\tau_I
=
\frac1h
\left(
F_{I-\frac12}-\widehat F_{I+\frac12}
\right)
-
\frac{h_\ell}{h}(f_I+C_I),
\]
and
\[
\tau_{I+1}
=
\frac1h
\left(
\widehat F_{I+\frac12}-F_{I+\frac32}
\right)
-
\frac{h_r}{h}(f_{I+1}+C_{I+1}).
\]

We first establish the individual estimates.

For the left interface node, Taylor expansion about
\(x=\alpha\) gives
\[
u_I
=
u^-(\alpha)
-\delta_\ell u_x^-(\alpha)
+\frac{\delta_\ell^2}{2}u_{xx}^-(\alpha)
+O(h^3),
\]
\[
u_{I-1}
=
u^-(\alpha)
-(\delta_\ell+h)u_x^-(\alpha)
+\frac{(\delta_\ell+h)^2}{2}u_{xx}^-(\alpha)
+O(h^3),
\]
and
\[
u_{I+1}
=
u^+(\alpha)
+\delta_r u_x^+(\alpha)
+\frac{\delta_r^2}{2}u_{xx}^+(\alpha)
+O(h^3).
\]
Using
\[
u^+(\alpha)=u^-(\alpha)+w,
\qquad
\beta^+u_x^+(\alpha)
=
\beta^-u_x^-(\alpha)+v,
\]
we obtain
\[
u_{I+1}-u_I
=
w
+
\left(
\delta_\ell+
\frac{\beta^-}{\beta^+}\delta_r
\right)u_x^-(\alpha)
+
\frac{\delta_r}{\beta^+}v
+
O(h^2).
\]
Hence
\[
\widehat F_{I+\frac12}
=
\frac{\widehat\beta_H}{h}
\left(
w+\frac{\delta_r}{\beta^+}v
\right)
+
\beta^-u_x^-(\alpha)
+
O(h),
\]
because
\[
\frac{\widehat\beta_H}{h}
\left(
\delta_\ell+
\frac{\beta^-}{\beta^+}\delta_r
\right)
=
\beta^-.
\]
Since \(x_{I-\frac12}\) lies \(O(h)\) from the interface,
\[
F_{I-\frac12}
=
\beta^-u_x^-(\alpha)+O(h).
\]
Therefore
\[
\frac1h
\left(
F_{I-\frac12}-\widehat F_{I+\frac12}
\right)
=
-\frac{\widehat\beta_H}{h^2}
\left(
w+\frac{\delta_r}{\beta^+}v
\right)
+
O(1).
\]
On the other hand,
\[
\frac{h_\ell}{h}C_I
=
-\frac{\widehat\beta_H}{h^2}
\left(
w+\frac{\delta_r}{\beta^+}v
\right).
\]
The potentially singular terms therefore cancel exactly, and since
\[
\frac{h_\ell}{h}f_I=O(1),
\]
we conclude that
\[
\tau_I=O(1).
\]

For the right interface node, the same argument is carried out from
the \(+\) side.  Using
\[
u^-(\alpha)=u^+(\alpha)-w,
\qquad
\beta^-u_x^-(\alpha)
=
\beta^+u_x^+(\alpha)-v,
\]
we obtain
\[
u_{I+1}-u_I
=
w
+
\left(
\delta_r+
\frac{\beta^+}{\beta^-}\delta_\ell
\right)u_x^+(\alpha)
-
\frac{\delta_\ell}{\beta^-}v
+
O(h^2).
\]
Hence
\[
\widehat F_{I+\frac12}
=
\frac{\widehat\beta_H}{h}
\left(
w-\frac{\delta_\ell}{\beta^-}v
\right)
+
\beta^+u_x^+(\alpha)
+
O(h),
\]
where
\[
\frac{\widehat\beta_H}{h}
\left(
\delta_r+
\frac{\beta^+}{\beta^-}\delta_\ell
\right)
=
\beta^+.
\]
Since \(x_{I+\frac32}\) lies \(O(h)\) from the interface,
\[
F_{I+\frac32}
=
\beta^+u_x^+(\alpha)+O(h).
\]
We have
\[
\frac1h
\left(
\widehat F_{I+\frac12}-F_{I+\frac32}
\right)
=
\frac{\widehat\beta_H}{h^2}
\left(
w-\frac{\delta_\ell}{\beta^-}v
\right)
+
O(1).
\]
Also,
\[
\frac{h_r}{h}C_{I+1}
=
\frac{\widehat\beta_H}{h^2}
\left(
w-\frac{\delta_\ell}{\beta^-}v
\right).
\]
Again the potentially singular terms cancel exactly.  Therefore
\[
\tau_{I+1}=O(1).
\]

Adding the two truncation errors eliminates the common cut-interval
flux exactly:
\begin{equation}
\label{eq:app-sum}
\tau_I+\tau_{I+1}
=
\frac1h
\left(
F_{I-\frac12}-F_{I+\frac32}
\right)
-
\frac1h
\left[
h_\ell(f_I+C_I)
+
h_r(f_{I+1}+C_{I+1})
\right].
\end{equation}

Let
\[
q(x):=\beta(x)u_x(x)
\]
denote the physical flux on each side of the interface.  On each
smooth interval, the centered numerical flux satisfies
\[
F_{i+\frac12}
=
q(x_{i+\frac12})+O(h^2).
\]
Consequently,
\[
\frac1h
\left(
F_{I-\frac12}-F_{I+\frac32}
\right)
=
\frac1h
\left(
q(x_{I-\frac12})-q(x_{I+\frac32})
\right)
+O(h).
\]

Since
\[
-q'(x)=f(x),
\]
we have
\[
q(\alpha^-)-q(x_{I-\frac12})
=
-\int_{x_{I-\frac12}}^\alpha f(x)\,dx,
\]
and
\[
q(x_{I+\frac32})-q(\alpha^+)
=
-\int_\alpha^{x_{I+\frac32}} f(x)\,dx.
\]
Using
\[
q(\alpha^+)=q(\alpha^-)+[\beta u_x]_\alpha,
\]
we obtain
\[
q(x_{I-\frac12})-q(x_{I+\frac32})
=
\int_{x_{I-\frac12}}^{x_{I+\frac32}}
f(x)\,dx
-
[\beta u_x]_\alpha.
\]
Hence
\begin{equation}
\label{eq:app-flux-difference}
\frac1h
\left(
F_{I-\frac12}-F_{I+\frac32}
\right)
=
\frac1h
\left(
\int_{x_{I-\frac12}}^{x_{I+\frac32}}
f(x)\,dx
-
[\beta u_x]_\alpha
\right)
+O(h).
\end{equation}

Since \(u^\pm\in C^3\) and \(\beta^\pm\) are constant,
\(f^\pm=-\beta^\pm u_{xx}^\pm\) are \(C^1\) up to the interface.
Hence
\[
\int_{x_{I-\frac12}}^\alpha f(x)\,dx
=
h_\ell f_I+O(h^2),
\]
and
\[
\int_\alpha^{x_{I+\frac32}} f(x)\,dx
=
h_r f_{I+1}+O(h^2).
\]
Therefore
\begin{equation}
\label{eq:app-source}
\int_{x_{I-\frac12}}^{x_{I+\frac32}}f(x)\,dx
=
h_\ell f_I+h_r f_{I+1}+O(h^2).
\end{equation}

Using the definitions of \(C_I\) and \(C_{I+1}\) above,
\[
h_\ell C_I
=
-\frac{\widehat \beta_H}{h}
\left(
[u]_\alpha
+
\frac{[\beta u_x]_\alpha}{\beta^+}
(x_{I+1}-\alpha)
\right),
\]
and
\[
h_r C_{I+1}
=
\frac{\widehat \beta_H}{h}
\left(
[u]_\alpha
+
\frac{[\beta u_x]_\alpha}{\beta^-}
(x_I-\alpha)
\right).
\]
Adding these identities, the \([u]_\alpha\)-terms cancel exactly.
Using
\[
\widehat \beta_H
=
\left(
\frac{x_{I+1}-\alpha}{\beta^+h}
+
\frac{\alpha-x_I}{\beta^-h}
\right)^{-1},
\]
we obtain
\begin{equation}
\label{eq:app-correction}
h_\ell C_I+h_rC_{I+1}
=
-[\beta u_x]_\alpha.
\end{equation}

Substituting
\eqref{eq:app-flux-difference},
\eqref{eq:app-source}, and
\eqref{eq:app-correction}
into \eqref{eq:app-sum}, we find
\[
\begin{aligned}
\tau_I+\tau_{I+1}
&=
\frac1h
\left(
\int_{x_{I-\frac12}}^{x_{I+\frac32}}f(x)\,dx
-
[\beta u_x]_\alpha
\right)
\\
&\quad
-
\frac1h
\left(
h_\ell f_I+h_rf_{I+1}
-
[\beta u_x]_\alpha
\right)
+O(h)
\\
&=
\frac1h\,O(h^2)+O(h)
=
O(h).
\end{aligned}
\]
Therefore
\[
\tau_I=O(1),
\qquad
\tau_{I+1}=O(1),
\qquad
\tau_I+\tau_{I+1}=O(h).
\]
This proves Lemma~\ref{lem:harmonic-cancellation}.
\section*{Declaration of Interest}

The authors declare that they have no known competing financial interests
or personal relationships that could have appeared to influence the work
reported in this paper.

\end{document}